\documentclass[11pt]{amsart}
\usepackage{amssymb,amsmath,amsopn,textcomp, thmtools}
\usepackage[foot]{amsaddr}
\usepackage[version=4]{mhchem}
\newtheorem{theorem}{Theorem}[section]
\newtheorem{lemma}{Lemma}[section]
\newtheorem{remark}{Remark}[section]
\newtheorem{example}{Example}[section]
\newtheorem{corollary}{Corollary}[section]
\newtheorem{proposition}{Proposition}[section]
\newcommand{\iR}{\mathbb{R}}
\newcommand{\iN}{\mathbb{N}}

\newcommand{\cH}{\mathcal{H}}

\newcommand{\cV}{\mathcal{V}}

\usepackage[shortlabels]{enumitem}

\usepackage{bbm}
\newcommand{\R}{\mathbb{R}}

\newcommand{\cT}{\mathcal{T}}

\newcommand{\cR}{\mathcal{R}}
\newcommand*{\norm}[1]{\left\Vert{#1}\right\Vert}
\newcommand*{\norms}[1]{\Vert{#1}\Vert}
\newcommand{\astT}{ \overset{T}{\ast}}

\usepackage{extarrows}
\usepackage{dsfont}
\usepackage{tikz,subcaption} 

\usepackage{mathabx} 

\allowdisplaybreaks

\begin{document}
	
\title{Duality estimates for subdiffusion problems including time-fractional porous medium type equations}

\author{Arl\'{u}cio Viana$^{a, b, \circ}$}
\thanks{$\circ$ A. V. thanks CAPES and the Humboldt foundation for the support through the CAPES-Humboldt Research Fellowship Programme for Experienced Researchers, Call 14/2022, number HUMBOLDT-22-PE32739979}
\email{arlucioviana@academico.ufs.br}

\author{Patryk Wolejko}
\email{patrykwolejko@hotmail.com}

\author{Rico Zacher$^{*,a}$}
\thanks{$^*$Corresponding author}
\email[Corresponding author:]{rico.zacher@uni-ulm.de}
\address{$^a$ Institut f\"ur Angewandte Analysis, Universit\"at Ulm, Helmholtzstra\ss{}e 18, 89081 Ulm, Germany.}

\address{$^b$ Departamento de Matem\'atica, Universidade Federal de Sergipe, Av. Marcelo Deda, 49107230 S\~ao Crist\'ov\~ao, Brazil.}


\begin{abstract}
	We prove duality estimates for time-fractional and more general subdiffusion problems. An important example is given by subdiffusive porous medium type equations. Our estimates can be used to prove uniqueness of weak solutions to such problems, and they allow to extend a key estimate from classical reaction-diffusion systems to the subdiffusive case. Besides concrete equations involving a Laplacian, we also consider abstract problems in a Hilbert space setting. 
\end{abstract}

\maketitle


\vspace{12pt}

\section{Introduction}
The main purpose of this paper is to derive duality estimates for equations whose prototypical
example is given by the subdiffusive porous medium type equation
\begin{equation} \label{IntPMED}
	\partial_t \big(k\ast [u-u_0]\big)-\Delta \big(a u^m\big)= f\quad\mbox{on}\;(0,T)\times \Omega
\end{equation}
with $m\ge 1$. Here $\Omega$ is a bounded domain in $\iR^N$ and the (nonnegative) unknown function $u$ also satisfies a homogeneous Dirichlet or Neumann boundary condition
as well as $u|_{t=0}=u_0$. By $k\ast v$ we mean the
convolution on the positive halfline $\iR_+:=[0,\infty)$ w.r.t.\ the time variable,
that is
$(k\ast v)(t)=\int_0^t k(t-\tau)v(\tau)\,d\tau$, $t\ge 0$. The kernel $k$ is assumed to satisfy
the condition
\begin{itemize}
\item [{\bf ($\mathcal{PC}$)}] $k\in L_{1,\,loc}(\iR_+)$ is nonnegative and nonincreasing, and there exists a kernel $l\in L_{1,\,loc}(\iR_+)$ such that
$k\ast l=1$ on $(0,\infty)$.
\end{itemize}
In this case, $k$ is said to be a $\mathcal{PC}$-kernel (cf. \cite{ZWH}) and we also write $(k,l)\in \mathcal{PC}$. Note that $(k,l)\in \mathcal{PC}$ implies that $l$ is completely
positive, see \cite[Theorem 2.2]{CN} and \cite{CP1}, in particular $l$ is nonnegative.
Observe that for sufficiently smooth $v$ with $v(0)=v_0$, 
\begin{equation} \label{operatorreform}
\partial_t \big(k\ast [v-v_0]\big)=k \ast \partial_t v.
\end{equation}

Condition ($\mathcal{PC}$) covers a wide spectrum of relevant integro-differential operators w.r.t.\ time that
appear in applications in physics in the context of subdiffusion processes, and it has already been considered in many works. The most prominent example is given by the pair $(k,l)=(g_{1-\alpha},g_\alpha)$ with $\alpha\in(0,1)$, where
$g_\beta$ denotes the standard (or Riemann-Liouville) kernel
\[
g_\beta(t)=\,\frac{t^{\beta-1}}{\Gamma(\beta)}\,,\quad
t>0,\quad\beta>0.
\]
With this choice, $\partial_t(k\ast v)$ becomes the classical Riemann-Liouville fractional derivative
$\partial_t^\alpha v$ of order $\alpha$, and $k \ast \partial_t v={}^c D_t^\alpha v$, the Caputo fractional derivative (cf.\ the right-hand side in (\ref{operatorreform})). Multi-term fractional derivatives and operators of distributed order also satisfy ($\mathcal{PC}$), see
\cite[Section 6]{VZ} and \cite{SaCa}, where further examples can be found.

The function $a$ in \eqref{IntPMED} is assumed to satisfy
\begin{equation} \label{Intacond}
	0<a_1\le a(t,x)\le a_2\quad\mbox{on}\;(0,T)\times \Omega,
\end{equation}
with two constants $a_1,a_2>0$. A crucial point in our analysis is that we aim for estimates
of $u$ in which $a$ only appears through the two constants $a_1$ and $a_2$. Our results provide in particular the estimate
\begin{equation} \label{MainEstimate}
		\int_0^T \int_\Omega  u^{m+1}\,dx\,dt
		 \le C(m) \Big(\frac{a_2}{a_1}\Big)^{m+1}\Big(T\norm{u_0}_{L_{m+1}(\Omega)}^{m+1}
		+\norm{l}_{L_1((0,T))}^{m+1}
		\norm{f}_{L_{m+1}(\Omega_T)}^{m+1}\Big),
	\end{equation}
see Theorem \ref{PMEthm} below. Estimate \eqref{MainEstimate} is the nonlocal (or subdiffusive) analogue of
a corresponding duality estimate (take formally $l=1$ in \eqref{MainEstimate}) for the problem 
\begin{equation} \label{localPDE}
\partial_t u-\Delta \big(a u^m\big)= f\quad\mbox{on}\;(0,T)\times \Omega,
\end{equation}
see \cite{LaPi}. The $L_2$-estimate for \eqref{localPDE} in the linear case $m=1$ seems to be better known and appears in many works, see e.g.\ \cite[Proposition 6.1]{Pie11} and \cite{LaPi} and the references 
given therein. It is a key estimate in the theory of systems of reaction-diffusion 
equations with quadratic reaction terms and control of mass as it shows the integrability of the nonlinear terms, thereby allowing to prove the global existence
of weak solutions in the case of arbitrary (different) diffusion coefficients of the involved species, see \cite{DeFePiVo}. It is now known that there are even global smooth solutions in this case (\cite{CaGoVa,FeMoTa}). The $L_{m+1}$-estimate for \eqref{localPDE} is equally useful when considering degenerate diffusion of porous medium type, cf.\ \cite{LaPi}. We refer to Section \ref{chem} for more details and further references.
There we also discuss applications of our duality estimates
in connection with subdiffusion-reaction systems, see Theorem
\ref{subdiffsystemthm} and Theorem \ref{entropyvar}.

Our proof of \eqref{MainEstimate} relies on estimates we derive earlier in Section \ref{dualityestimates} on pairs of functions
$(u,w)$ subject to
\begin{equation} \label{BasicPDE}
\partial_t \big(k\ast [u-u_0]\big)-\Delta w= f.
\end{equation}
These basic estimates can also be used to obtain uniqueness of weak solutions for equations such as
\[
\partial_t \big(k\ast [u-u_0]\big)-\Delta\big( a \Phi(u)\big)= f,
\]
where $a$ is as above and $\Phi:\;\iR\to \iR$ is continuous and strictly increasing with $\Phi(0)=0$, see Theorem \ref{uniquetheorem} below. This result and its proof significantly generalize the considerations in \cite[Section 2, Theorem 1]{LoOPR} on a time-fractional porous medium type equation in one space dimension, see also Remark \ref{PloRem}. In addition to considering a more general equation with a $\mathcal{PC}$-kernel $k$, our proofs are also rigorous. In particular, a suitable time regularization of the weak formulations is explained in detail.

As we show in Section 7, the results for \eqref{BasicPDE} can be extended to abstract problems in a Hilbert space setting where instead of the Laplacian we have an operator induced by a symmetric bilinear form. In contrast to the Laplacian case, where we also
allow subsolutions, the results in the abstract setting are restricted to solutions. As an example, we consider fully nonlocal porous medium type equations in the whole space given by
\begin{equation} \label{2NPMED_Intro}
	\partial_t \big(k\ast [u-u_0]\big) + (-\Delta)^{\beta/2} \big(a u^m\big)= f,\quad (t,x)\in (0,T)\times\R^N ,
\end{equation}
with $\beta\in(0,2]$. We find new estimates in both cases $m\geq1$ and $0<m\leq1$, under suitable conditions.

Time-fractional and more general subdiffusion equations of porous medium type (including problems with nonlocal diffusion term) have recently been investigated by many authors and different methods, see e.g.\ \cite{Aka,ACV2,DiValVes,DNA19,LRS18,LoOPR,Plo15,SchmWitt,VZ,WWZ}. 
Abstract nonlinear and nonlocal in time equations with a $\mathcal{PC}$-kernel in a Hilbert space framework are investigated in \cite{Aka}. Here, the nonlinearity is in
subdifferential form, and as an application, the author obtains an existence and uniqueness result for porous medium type problems under very mild assumptions on the data. 
Abstract (stochastic) nonlinear time-fractional and more general subdiffusion problems with explicit time dependence in the nonlinearity are studied in \cite{LRS18,LRS19} by means of monotonicity methods. Among others, the authors are able to prove existence and uniqueness for deterministic and stochastic time-fractional porous medium equations.
In a series of papers (see \cite{JakWitt,Sap,SchWitt}), existence and uniqueness of entropy solutions to doubly nonlinear equations with an integro-differential operator as in 
\eqref{IntPMED} and $k$ from a rather general class of kernels (including the standard kernel) are studied by using tools from the theory of accretive operators.
This approach relies on results for abstract nonlinear Volterra equations \cite{Grip2,Grip1} (see also \cite{CNa,CN}) and also works for nonlinear terms with explicit space dependence.
Degenerate problems of the form
\begin{equation} \label{problemWWZ}
\partial_t\big(k\ast[u-u_0]\big)-\mbox{div}\big(A\nabla \varphi(u)\big) = f,
\end{equation}
with a $\mathcal{PC}$-kernel $k$ and a merely bounded, measurable and uniformly elliptic
coefficient matrix $A=A(t,x)$, are considered in \cite{SchmWitt,WWZ}. In \cite{WWZ}, existence of bounded weak solutions is established under rather weak assumptions on the data. Further, an $L_1$-contraction principle is derived. Based on the results in \cite{WWZ}, existence of entropy solutions to \eqref{problemWWZ} for general $L_1$-data is shown in \cite{SchmWitt}. Concerning the long-time behaviour, decay estimates for time-fractional porous medium equations are derived in
\cite{AfVal,DiValVes,VZ}. Finally, numerical schemes for such equations are investigated in \cite{LoOPR,Plo15,Plo14}.

The paper is organized as follows. Section 2 collects some auxiliary tools. We state various useful properties of the dual convolution on an interval $(0,T)$ and recall a fundamental identity of integro-differential operators of the form $\frac{d}{dt}(k\ast \cdot)$. We also
consider the Yosida approximation of such operators and explain how it leads to a useful regularization of the kernel $k$. In Section 3, we prove a general duality estimate for the equation \eqref{BasicPDE}. This estimate is then applied to the porous medium type problem \eqref{IntPMED} in Section 4. Section 5 is concerned with uniqueness of weak solutions.
In Section 6, we apply our basic duality estimate to systems of subdiffusion-reaction systems. In Section 7, we extend the basic duality estimate (for solutions) to the abstract Hilbert space setting and discuss as an example a fully nonlocal porous medium type problem.

\section{Preliminaries} \label{Prelim}
Let $\mathcal{H}$ be a real Hilbert space with inner product $(\cdot, \cdot)_\mathcal{H}$ and $T>0$. For $f\in L_1((0,T))$ and $g\in L_1((0,T); \mathcal{H})$ we define
\[
(f \astT g)(t):=\int_t^T f(\tau-t)g(\tau)\,d\tau,\quad t\in (0,T).
\]
As the following proposition shows, $\astT$ can be regarded as the dual convolution with respect to the interval $(0,T)$.
\begin{proposition}
	Let $1\le p\le \infty$, $b, h\in L_1((0,T))$ and $k\in H^1_1((0,T))$. Then the following properties hold true.
	\begin{enumerate}[(a)] 
		\item \textit{Duality property:} For every $f\in L_p((0,T);\cH)$ and $g\in L_{p'}((0,T);\cH)\;(\frac{1}{p}+\frac{1}{p'}=1)$
		\begin{equation} \label{dualprop}
			\int_0^T \big( (h\ast f)(t) , g(t) \big)_\mathcal{H}\,dt=\int_0^T \left( f(t) ,(h\astT g)(t)\right)_\mathcal{H}\,dt.
		\end{equation}
		\item \textit{Successive dual convolution property:} For every $g\in L_1((0,T);\cH)$ we have
		\begin{equation}\label{succprop}
			b\astT (h \astT g)=(b\ast h)\astT g\quad \;\;\mbox{in}\;L_1((0,T);\cH).
		\end{equation}
		\item  \textit{Duality property of integro-differential operators:} For every $f\in L_p((0,T);\cH)$ and $g\in L_{p'}((0,T);\cH)$,
		\begin{equation} \label{dualpropp}
			\int_0^T \left( \frac{d}{dt}(k\astT f)(t) ,  g(t) \right)_\mathcal{H}\,dt=-\int_0^T \left( f(t) ,  \frac{d}{dt}(k\ast g)(t) \right)_\mathcal{H} \,dt.
		\end{equation}
	\end{enumerate}	
\end{proposition}
\begin{proof}
	The duality property can be easily verified using Fubini's theorem. Indeed,
	\begin{align*}
		\int_0^T \big( (h\ast f)(t) , g(t) \big)_\mathcal{H}\,dt & = \int_0^T \left( \int_0^t h(t-\tau)f(\tau)\,d\tau ,  g(t) \right)_\mathcal{H} \,dt\\
		& =  \int_0^T \int_\tau^T h(t-\tau) \big( f(\tau) , g(t) \big)_\mathcal{H} \,dt\,d\tau \\
		& =\int_0^T \left( f(\tau) , (h\astT g)(\tau)\right)_\mathcal{H}\,d\tau.
	\end{align*}
	
	To see the second property, let $\psi\in C([0, T]; \mathcal{H})$ be an arbitrary function. Then applying the duality property several times and using that $\ast$ is associative, we have 
	\begin{align*}			
		\int_0^T \left(\psi(t),\big(b\astT(h \astT g)\big)(t)\right)_\mathcal{H}\,dt & =\int_0^T \left( (b\ast \psi)(t) , (h \astT g)(t) \right)_\mathcal{H}\, dt \\
		& =\int_0^T \big(((h * b) * \psi)(t), g(t)\big)_\mathcal{H}\, dt \\
		& =\int_0^T \left(\psi(t) , ((h \ast b) \astT g) (t)\right)_\mathcal{H}\,dt.
	\end{align*}
Clearly, $h\ast b=b\ast h$. Taking $\psi(t)=\varphi(t)x$ with $\varphi\in C([0,T])$ and $x\in \cH$, it thus follows that
\[
\int_0^T \varphi(t) \big(x,\big(b\astT(h \astT g)\big)(t)\big)_\mathcal{H}\,dt=
\int_0^T \varphi(t) \big(x, ((b \ast h) \astT g) (t)\big)_\mathcal{H}\,dt.
\]
The fundamental lemma of the calculus of variations now implies that
\[
 \big(x,\big(b\astT(h \astT g)\big)(t)\big)_\mathcal{H}=\big(x, ((b \ast h) \astT g) (t)\big)_\mathcal{H}\quad \mbox{for a.a.}\;t\in (0,T)\;\mbox{and all}\;x\in \cH,
\]
from which we can immediately deduce the asserted identity. 

Finally, we prove \eqref{dualpropp}. Using \eqref{dualprop} and the fact that $k\in H_1^1((0,T))$ we have
\begin{align*}
		\int_0^T &\left( \frac{d}{dt}(k\astT f)(t) , g(t) \right)_\mathcal{H}\,dt  = \int_0^T \left( \frac{d}{dt} \left( \int_t^T  k(\tau-t) f(\tau) d \tau\right), g(t) \right)_\mathcal{H} d t \\
		&\quad =  \int_0^T \left(  \left( \int_t^T[-\dot{k}(\tau-t)] f(\tau) d \tau-k(0) f(t)\right), g(t) \right)_\mathcal{H} d t \\
		&\quad =  \int_0^T \left[ \left(-(\dot{k} \astT f)(t) , g(t)\right)_\mathcal{H}  -k(0) \big(f(t) , g(t) \big)_\mathcal{H} \right] d t \\
		&\quad =- \int_0^T \left[ \left(f(t) , (\dot{k}\ast g)(t)\right)_\mathcal{H}  +k(0) \big(f(t) , g(t) \big)_\mathcal{H} \right] d t \\
		& \quad= -\int_0^T  \left(f(t) , \frac{d}{dt}(k\ast g)(t)\right)_\mathcal{H}  d t .
	\end{align*}
\end{proof}

The next statement can be already found in \cite[Lemma 2.2]{VeZa08}.
\begin{lemma} \label{fundlemma1}
	Let $\cH$ be a real Hilbert space and $T>0$. Then for any
	$k\in H^1_1((0,T))$ and any $v\in L_2((0,T);{\cH})$ there holds
	\begin{align}
		\left( v(t),\frac{d}{dt}\,(k\ast v)(t)\right)_{{\cH}} =&\;
		\frac{1}{2}\,\frac{d}{dt}\,(k\ast
		\norm{v}_{\cH}^2)(t)+\frac{1}{2}\,k(t)\norm{v(t)}_{\cH}^2 \nonumber\\
		&\;+\,\frac{1}{2}\,\int_0^t [-\dot{k}(s)]\,\norm{v(t)-v(t-s)}_{\cH}^2\,ds,\quad
		\mbox{a.a.}\;t\in(0,T).\label{ident1}
	\end{align}
\end{lemma}

Lemma \ref{fundlemma1} provides a fundamental identity for integro-differential operators of the
form $\frac{d}{dt}(k\ast \cdot)$ provided that $k\in H^1_1((0,T))$. In particular, it does not apply
to the Riemann-Liouville fractional derivative, unless the function $v$ is much more regular. In order to use \eqref{ident1} for kernels of type $\mathcal{PC}$ one can regularize the kernel by replacing 
$\frac{d}{dt}(k\ast \cdot)$ with its Yosida approximation. This regularization method goes back to
\cite{VeZa08} (see the proof of Theorem 2.1) and \cite{Za-08,ZWH} and is now well established.
We will also use it in the present paper.

To explain this regularization approach, we first summarize some properties of kernels of type $\mathcal{PC}$. Let $(k,l)\in \mathcal{PC}$.
For $\gamma\ge 0$, the kernels $s_\gamma, r_\gamma \in L_{1,loc}(\iR_+)$ are defined by the scalar Volterra equations (cf.\ \cite{GLS})
\begin{align*}
s_\gamma(t)+\gamma(l\ast s_\gamma)(t) & = 1,\quad t>0,\\
r_\gamma(t)+\gamma(l\ast r_\gamma)(t) & = l(t),\quad t>0.
\end{align*}
Both $s_\gamma$ and $r_\gamma$ are nonnegative for all $\gamma\ge 0$, by complete positivity of $l$ (see \cite{CN}, \cite{JanI}). Furthermore, the relaxation function $s_\gamma$ belongs to $H^1_{1,\,loc}(\iR_+)$ and is nonincreasing for all 
$\gamma\ge 0$.

For $\gamma>0$, we denote by $h_\gamma\in L_{1,loc}(\iR_+)$ the resolvent kernel associated
with $\gamma l$ (cf.\ \cite{GLS}), which means that
\begin{equation} \label{hndef}
h_\gamma(t)+\gamma(h_\gamma\ast l)(t)=\gamma l(t),\quad t>0.
\end{equation}
Note that $h_\gamma=\gamma r_\gamma=-\dot{s}_\gamma \in L_{1,\,loc}(\iR_+)$, which also shows that $h_\gamma$ is nonnegative. Further, it is well-known that for any $f\in L_p((0,T))$, $1\le p<\infty$, we have
$h_n\ast f\rightarrow f$ in $L_p((0,T))$ as $n\rightarrow \infty$, see e.g.\ \cite{Za-08}.

Setting
\begin{equation} \label{kndef}
k_{\gamma}=k\ast h_\gamma,\quad \gamma>0,
\end{equation}
it is known that $k_\gamma=\gamma s_\gamma$ (see e.g.\ \cite{Za-08}), and hence
$k_\gamma$ is also nonnegative and nonincreasing, and it lies in
$H^1_{1,\,loc}(\iR_+)$ as well. The kernels $k_\gamma$, $\gamma>0$,  approximate the kernel $k$ in the following sense. Letting $X$ be a real Banach space, the operator $B$ defined by
\[ B u=\frac{d}{dt}(k\ast u),\;\;D(B)=\{u\in L_p((0,T);X):\,k\ast u\in \mbox{}_0 H^1_p((0,T);X)\},
\]
where the zero means vanishing at $t=0$, is known to be
$m$-accretive in $L_p((0,T);X)$, see
\cite{Grip1}. Its Yosida approximations $B_{n}$ are defined by
$B_{n}=nB(n+B)^{-1},\,n\in \iN$. For every
$u\in D(B)$, $B_{n}u\rightarrow Bu$ in $L_p((0,T);X)$ as
$n\to \infty$. Moreover,
\begin{equation} \label{Yos}
B_n u=\frac{d}{dt}(k_n\ast u),\quad u\in L_p((0,T);X),\;n\in
\iN,
\end{equation}
see \cite{Za-08}, so that the kernels $k_n$ take on the role of $k$ when $B$ is replaced by its Yosida approximations $B_n$.
   
Virtually all important examples of  $\mathcal{PC}$-kernels (e.g.\ fractional derivative (also with exponential shift), multi-term fractional derivative, distributed order case) are completely monotone, see \cite[Section 6]{VZ}. This means that $k\in C^\infty((0,\infty))$ and
$(-1)^j k^{(j)}\ge 0$ on $(0,\infty)$ for all $j\in \iN_0$. In this case, we can say much more about $l$, $s_\gamma$, $h_\gamma$ and $k_\gamma$. In fact, since $k\ast l=1$, we have $k(0+)=\infty$
and thus \cite[Thm.\ 5.4, p. 159]{GLS} shows that $l$ is also completely monotone. The latter property 
implies that for every $\gamma>0$ the resolvent kernel of $\gamma l$, that is $h_\gamma$, is completely monotone, too, see \cite[Thm.\ 3.1, p. 148]{GLS}. Since $h_\gamma=-\dot{s}_\gamma$ (see above), it follows that $s_\gamma$ is also completely monotone, for all $\gamma>0$. 
This implies in particular that
$k_\gamma=\gamma s_\gamma$ is convex in $(0,\infty)$ for all $\gamma>0$.

We record two further properties (which should be well-known) in the following proposition. For the reader's convenience we provide a proof. Some ideas are contained in \cite{KRZ,SaCa}.
\begin{proposition} \label{comoprop}
Let $k$ be a completely monotone $\mathcal{PC}$-kernel. Then
\begin{itemize}
\item[(i)] $\lim_{t\to 0+}tk(t)=0$.
\item[(ii)] The function $\{t\to t\dot{k}(t)\}$ is locally integrable on $[0,\infty)$.
\end{itemize}
\end{proposition}
\begin{proof}
(i) We know that $l$ is also completely monotone. In particular, $k$ and $l$ are nonincreasing and thus
\[
1=(k\ast l)(t)\ge k(t)l(t) (1\ast 1)(t)=t k(t)l(t),\quad t>0,
\]
which implies
\[
t k(t)\le \frac{1}{l(t)},\quad t>0.
\]
Since $l(0+)=\infty$, it follows that $\lim_{t\to 0+}tk(t)=0$.

(ii) Convexity of $k$ gives
\[
k(\frac{t}{2}\big)\ge k(t)+\dot{k}(t)\big(-\frac{t}{2}\big),\quad t>0,
\]
and hence $0\le -t \dot{k}(t)\le 2k(t/2)$ for all $t>0$, by positivity of $k$ and $-\dot{k}$. The assertion then follows from
$k\in L_{1,loc}(\iR_+)$.
\end{proof}
\section{The basic duality estimate with Laplacian} \label{dualityestimates}
Let $T>0$ and $\Omega\subset \iR^N$ be a (sufficiently smooth) bounded domain. Set $\Omega_T=(0,T)\times \Omega$. We want to study inequalities of the form
\begin{equation} \label{SubProb}
	\partial_t \big(k\ast [u-u_0]\big)-\Delta w\le f,\quad (t,x)\in \Omega_T,
\end{equation}
in a weak setting. Here $u$ and $w$ are the functions of interest, $k$ is a $\mathcal{PC}$-kernel
and $u_0$ (initial value for $u$) and $f$ are given data. Later, we will consider various situations
where $w$ is expressed in terms of $u$. Concerning boundary conditions, we consider the case
of a Dirichlet boundary condition, $w=0$ on $ (0,T)\times \partial \Omega$, and the case
of a Neumann boundary condition, $\partial_\nu w=0$ on $ (0,T)\times \partial \Omega$.

The weak formulation of \eqref{SubProb} reads as follows. The pair $(u,w)$ satisfies 
\eqref{SubProb} in the weak sense if 
\begin{equation} \label{weakform}
	\int_0^T\int_{\Omega}\Big( -\eta_t \big(k\ast [u-u_0]\big)+\nabla w\cdot \nabla \eta\Big)\,dx\,dt\le
	\int_0^T\int_{\Omega} f\eta\,dx\,dt
\end{equation}
for all nonnegative test functions $\eta\in \cT$ satisfying $\eta|_{t=T}=0$, where 
\begin{align*}
\cT&=\mathring{H}^{1,1}_2(\Omega_T):= H^1_2((0,T);L_2(\Omega))\cap
L_2((0,T);\mathring{H}^1_2(\Omega))\;\;\mbox{in the Dirichlet case},\\
\cT&={H}^{1,1}_2(\Omega_T):= H^1_2((0,T);L_2(\Omega))\cap
L_2((0,T);{H}^1_2(\Omega))\;\;\mbox{in the Neumann case}.
\end{align*}
Here $\mathring{H}^1_2(\Omega):=\overline{C_0^\infty(\Omega)}\,{}^{H^1_2(\Omega)}$.
The functions $u$ and $w$ and the data are assumed to lie in suitable spaces so that \eqref{weakform} makes sense. We will be more specific in Lemma \ref{BasicEstimate1} below.
In any case, we assume that $(k\ast u)|_{t=0}=0$ so that no additional term appears when integrating by parts in time.

In order to derive {\em a priori} estimates for $(u,w)$, it is expedient to reformulate \eqref{weakform} 
in such a way that $k$ is replaced by the more regular kernel $k_n$ ($n\in \iN$). This can be achieved by adapting the ideas from Lemma 3.1 in \cite{Za-08}. We choose
\[
\eta=h_n\astT \varphi,
\]
where $\varphi\in \cT$ is nonnegative with $\varphi|_{t=T}=0$ and $h_n$ ($n\in \iN$) is the kernel introduced in Section 
\ref{Prelim}. Observe that $\eta|_{t=T}=0$ and $\eta_t=h_n\astT \varphi_t$. Using the duality property \eqref{dualprop} and $k_n=k\ast h_n$ and integrating by parts we then have
\begin{align*}
	-\int_0^T\int_{\Omega} &\eta_t \big(k\ast [u-u_0]\big)\,dx\,dt =-\int_0^T\int_{\Omega} 
	\big(h_n\astT \varphi_t\big) \big(k\ast [u-u_0]\big)\,dx\,dt\\
	&  =-\int_0^T\int_{\Omega} 
	\varphi_t \big(k_n\ast [u-u_0]\big)\,dx\,dt=\int_0^T\int_{\Omega} 
	\varphi \partial_t\big(k_n\ast [u-u_0]\big)\,dx\,dt.
\end{align*}
Further
\[
\int_0^T\int_{\Omega}\nabla w\cdot \nabla \eta\,dx\,dt=\int_0^T\int_{\Omega}
(h_n\ast\nabla w)\cdot \nabla \varphi\,dx\,dt
\]
and
\[
\int_0^T\int_{\Omega} f\eta\,dx\,dt=\int_0^T\int_{\Omega}(h_n\ast f)\varphi\,dx\,dt.
\]
We may then argue as in \cite[Lemma 3.1]{Za-08} to obtain that for every nonnegative $\psi\in \mathring{H}^1_2(\Omega)$ (Dirichlet case) and $\psi\in {H}^1_2(\Omega)$ (Neumann case), respectively, there holds
\begin{align} \label{weakform2}
	\int_\Omega \Big(\psi \partial_t\big(k_n\ast [u-u_0]\big)+(h_n\ast\nabla w)\cdot \nabla \psi\Big)\,dx\le \int_\Omega (h_n\ast f)\psi\,dx,\quad \mbox{a.a.}\;t\in (0,T),
\end{align}
for all $n\in \iN$. Note that the converse is also true, that is, \eqref{weakform2} implies \eqref{weakform}.
\begin{lemma} \label{BasicEstimate1}
	Let $u_0\in L_2(\Omega)$, $f\in L_{2}(\Omega_T)$, $u\in L_2(\Omega_T)$ and $k\ast u\in C([0,T];L_2(\Omega))$ such that $(k\ast u)|_{t=0}=0$.
	Assume that one of the following conditions is satisfied.
	\begin{itemize}
		\item[(D)] $w\in L_2((0,T);\mathring{H}^1_2(\Omega))$, $w\ge0$ and \eqref{weakform2} holds for all $\psi\in \mathring{H}^1_2(\Omega)$ with $\psi\ge0$.
		\item[(N)] $w\in L_2((0,T);{H}^1_2(\Omega))$, $w\ge 0$ and \eqref{weakform2} holds for all $\psi\in{H}^1_2(\Omega)$ with $\psi\ge0$.
	\end{itemize}
	Then
	\begin{align}
		& \int_0^T \int_\Omega w (h_n\ast u)\,dx\,dt+
		\frac{1}{2}\,\int_0^T \int_\Omega \big[k_n(T-t)+ k_n(t)\big]| l \astT \nabla w|^2\,dx\,dt
		\nonumber\\
		&+\frac{1}{2}\int_0^T \int_0^t\int_\Omega [-\dot{k}_n(s)]
		\big|(l \astT \nabla w)(t,x)-(l \astT \nabla w)(t-s,x)\big|^2\,dx\,ds\,dt\nonumber\\
		& \le \int_0^T \int_\Omega w u_0 (1\ast h_n)\,dx\,dt+\int_0^T \int_\Omega (h_n\ast f)(l\astT w) \,dx\,dt\nonumber\\
		& \quad+\int_0^T \int_\Omega \big[(h_n\astT \nabla w)-(h_n\ast\nabla w)\big] \cdot (l \astT \nabla w)\,dx\,dt. \label{BasEst1}
	\end{align}
\end{lemma}
\begin{proof}
	For $t\in (0,T)$ we choose in \eqref{weakform2}
	\[
	\psi=(l \astT w)(t,\cdot)
	\]
	and integrate in time over $(0,T)$.
	
	By the duality property \eqref{dualprop},
	\begin{align*}
		\int_0^T \int_\Omega (l \astT w) \partial_t\big(k_n\ast [u-u_0]\big)\,dx\,dt=
		\int_0^T \int_\Omega w \big(l\ast \partial_t\big(k_n\ast [u-u_0]\big)\big)\,dx\,dt.
	\end{align*}
	Note that
	\[
	l\ast \partial_t\big(k_n\ast [u-u_0]\big)=\partial_t \big(l\ast k\ast h_n\ast [u-u_0]\big)=h_n\ast [u-u_0].
	\]
	Thus it follows that
	\begin{align} \label{Term1}
		\int_0^T \int_\Omega (l \astT w) \partial_t\big(k_n\ast [u-u_0]\big)\,dx\,dt=
		\int_0^T \int_\Omega w \big(h_n\ast [u-u_0]\big)\,dx\,dt.
	\end{align}
	
	Next, we write
	\begin{align*}
		\int_0^T \int_\Omega (h_n\ast\nabla w)\cdot \nabla (l \astT w)\,dx\,dt=
		\int_0^T \int_\Omega (h_n\astT\nabla w)\cdot (l \astT \nabla w)\,dx\,dt+\cR_n,
	\end{align*}
	where 
	\[
	\cR_n=\int_0^T \int_\Omega \big[(h_n\ast\nabla w)-(h_n\astT\nabla w)\big] \cdot (l \astT \nabla w)\,dx\,dt.
	\]
	
	Using the identity $f=-\partial_t (1\astT f)$ and the successive dual convolution property we find that
	\begin{align}
		\int_0^T \int_\Omega (h_n\astT\nabla w)& \cdot (l \astT \nabla w)\,dx\,dt
		= -\int_0^T \int_\Omega \partial_t \big[(k\ast l)\astT (h_n\astT\nabla w)\big] \cdot (l \astT \nabla w)\,dx\,dt\nonumber\\
		&=  -\int_0^T \int_\Omega \partial_t \big[(k\ast h_n)\astT (l \astT\nabla w)\big] \cdot (l \astT \nabla w)\,dx\,dt\nonumber\\
		&=  \int_0^T \int_\Omega (l \astT\nabla w) \cdot \partial_t \big[k_n\ast (l \astT \nabla w)\big]\,dx\,dt, \label{term2a}
	\end{align}
	where in the last step we employ the duality relation \eqref{dualpropp}. By Lemma \ref{fundlemma1},
	\begin{align}
		\int_\Omega & (l \astT\nabla w) \cdot  \partial_t \big[k_n\ast (l \astT \nabla w)\big]\,dx\nonumber\\
		& = \frac{1}{2}\,\frac{d}{dt}\Big(k_n\ast \norms{l \astT \nabla w}^2_{L_2(\Omega)}\Big)(t)+
		\frac{1}{2}\,k_n(t)\norms{l \astT \nabla w}^2_{L_2(\Omega)}(t)\nonumber\\
		& \quad+\frac{1}{2}\int_0^t [-\dot{k}_n(s)]
		\norm{(l \astT \nabla w)(t,\cdot)-(l \astT \nabla w)(t-s,\cdot)}^2_{L_2(\Omega)}ds,\quad \mbox{a.a.}\,t\in (0,T). \label{term2b}
	\end{align}
	Combining \eqref{term2a} and \eqref{term2b} yields
	\begin{align}
		\int_0^T \int_\Omega (h_n\astT\nabla w)& \cdot (l \astT \nabla w)\,dx\,dt =
		\frac{1}{2}\,\int_0^T \big[k_n(T-t)+ k_n(t)\big]\norms{l \astT \nabla w}^2_{L_2(\Omega)}(t)\,dt
		\nonumber\\
		& \quad+\frac{1}{2}\int_0^T \int_0^t [-\dot{k}_n(s)]
		\norm{(l \astT \nabla w)(t,\cdot)-(l \astT \nabla w)(t-s,\cdot)}^2_{L_2(\Omega)}ds\,dt.
		\label{term2c}
	\end{align}
	This proves \eqref{BasEst1}. 
\end{proof}

\begin{corollary} \label{BasicEstimate2}
	Suppose that the assumptions from Lemma \ref{BasicEstimate1} are satisfied. Then
	\begin{align}
		\int_0^T \int_\Omega w u\,dx\,dt+&
		\frac{1}{2}\,\int_0^T \int_\Omega \big[k(T-t)+ k(t)\big]| l \astT \nabla w|^2\,dx\,dt\nonumber\\
		& \le \int_0^T \int_\Omega \big(u_0+l\ast f\big) w \,dx\,dt.
		\label{BasEst2}
	\end{align}
\end{corollary}
\begin{proof}
	The assertion follows directly from \eqref{BasEst1} by first dropping the (nonnegative) triple integral term, then sending $n\to \infty$ and finally using the duality property \eqref{dualprop}.
\end{proof}

\begin{remark} {\em 
    In the proof of Lemma \ref{BasicEstimate1}, the elliptic term can also be treated in a different way. Applying the duality identity \eqref{dualprop}, we have
    \[ \int_0^T \int_\Omega (h_n\ast\nabla w)\cdot \nabla (l \astT w)\,dx\,dt  = \int_0^T \int_\Omega (l\ast h_n\ast\nabla w)\cdot \nabla w\,dx\,dt  .\]
    Then, sending $n\to\infty$, as in the proof of Corollary \ref{BasicEstimate2}, yields the estimate
    \begin{equation}
	 \int_0^T \int_\Omega w u\,dx\,dt +
        \int_0^T \int_\Omega (l\ast\nabla w)\cdot \nabla w\,dx\,dt   \le \int_0^T \int_\Omega \big(u_0+l\ast f\big) w \,dx\,dt .
        \label{BasEst1.1}
	\end{equation}
}
    
\end{remark}

\medskip

In Corollary \ref{BasicEstimate2}, we did not exploit the nonnegative triple integral term from \eqref{BasEst1}. In order to obtain an estimate from it we are looking for a lower bound of
$-\dot{k}_n$ which allows to take the limit as $n\to \infty$.    

Let us assume that $k$ is completely monotone, a property that is satisfied by virtually all important examples of $\mathcal{PC}$-kernels. Then $k_n$ is convex in $(0,\infty)$ for all $n\in \iN$, see the end of Section \ref{Prelim}. For $t>0$ we therefore have
\[
k_n(2t)\ge k_n (t)+\dot{k}_n(t) t,
\]
which yields the lower bound
\begin{equation} \label{knbound}
	-\dot{k}_n(t)\ge \frac{1}{t} \big(k_n(t)-k_n(2t)\big),\quad t>0.
\end{equation}
Note that the right-hand side in \eqref{knbound} is positive.

These considerations lead to the following result.
\begin{corollary} \label{BasicEstimate3}
	Suppose that the assumptions from Lemma \ref{BasicEstimate1} are satisfied and assume in addition that $k$ is completely monotone. Then
	\begin{align}
		&\int_0^T \int_\Omega w u\,dx\,dt+
		\frac{1}{2}\,\int_0^T \int_\Omega \big[k(T-t)+ k(t)\big]| l \astT \nabla w|^2\,dx\,dt\nonumber\\
		& +\frac{1}{2}\int_0^T \int_0^t\int_\Omega \frac{1}{s} \big(k(s)-k(2s)\big)
		\big|(l \astT \nabla w)(t,x)-(l \astT \nabla w)(t-s,x)\big|^2\,dx\,ds\,dt\nonumber\\
		& \le \int_0^T \int_\Omega \big(u_0+l\ast f\big) w \,dx\,dt.
		\label{BasEst3}
	\end{align}
\end{corollary}
\begin{proof}
	The assertion is a consequence of Lemma \ref{BasicEstimate1}, the lower bound \eqref{knbound}
	and the fact that $k_n\to k$ in $L_1((0,T))$ as $n\to \infty$. In fact, the latter implies that for a suitable subsequence $(n_j)_{j\in \iN}$ we have pointwise convergence a.e.\ of $k_{n_j}$ to $k$ on
	$(0,T)$ and hence we can apply Fatou's lemma to the integral
	\[
	\int_0^T \int_0^t\int_\Omega \frac{1}{s} \big(k_{n_j}(s)-k_{n_j}(2s)\big)
	\big|(l \astT \nabla w)(t,x)-(l \astT \nabla w)(t-s,x)\big|^2\,dx\,ds\,dt.
	\]
\end{proof}
\begin{example}\label{PME_frac}
	{\em
		Let us consider the pair of standard kernels $(k,l)=(g_{1-\alpha},g_\alpha)$ with $\alpha\in (0,1)$.
		Then
		\[
		k(t)-k(2t)=\frac{t^{-\alpha}}{\Gamma(1-\alpha)}\big(1-2^{-\alpha}\big)=:c(\alpha)t^{-\alpha},\quad t>0.
		\]
		and thus the triple integral term in \eqref{BasEst3} takes the form
		\begin{equation}
			\frac{c(\alpha)}{2} \int_0^T \int_0^t\int_\Omega s^{-\alpha-1}
			\big|(g_\alpha \astT \nabla w)(t,x)-(g_\alpha \astT \nabla w)(t-s,x)\big|^2\,dx\,ds\,dt.
		\end{equation}
Using \cite[Lemma 2.2]{AlCaVa}, the last term, together with the
second term on the left-hand side of \eqref{BasEst3}, gives a seminorm estimate for $g_\alpha\astT \nabla w$ in $H^\frac{\alpha}{2}((0,T);L_2(\Omega))$.

        Alternatively, we can use \eqref{BasEst1.1} with $l=g_\alpha$, together with the inequality
        \begin{equation} \label{galphaInequality}
        \int_0^T v(g_\alpha\ast v)(t) \, dt \geq \gamma_\alpha \int_0^T |(g_{\alpha/2}\ast v)(t)|^2 dt ,\end{equation}
        $\alpha\in (0,1)$, $\gamma_\alpha=\cos(\frac{\alpha \pi}{2})$, see \cite[Lemma 3.1]{MuSchö}, to find that
        \begin{equation}
	 \int_0^T \int_\Omega w u\,dx\,dt +
       \gamma_\alpha \int_0^T \int_\Omega |g_{\alpha/2}\ast \nabla w|^2 \,dx \,dt  \le \int_0^T \int_\Omega \big(u_0+l\ast f\big) w \,dx\,dt .
        \label{BasEst1.2}
	\end{equation}
    The second term in \eqref{BasEst1.2} yields a seminorm estimate for $w$ in the space $H^{-\alpha/2}_2((0,T);H^1_2(\Omega))$, which, in terms of differentiation orders, is comparable to the estimate $g_\alpha\astT \nabla w$ in $H^\frac{\alpha}{2}((0,T);L_2(\Omega))$.
	}
\end{example}
\begin{remark}
{\em
Note that in the weak formulation of the equation 
\begin{equation} \label{SubProbS}
	\partial_t \big(k\ast [u-u_0]\big)-\Delta w= f,\quad (t,x)\in \Omega_T,
\end{equation}
we do not require the nonnegativity of the test function $\eta$. Hence, for weak solutions $(u,w)$, the estimates from Lemma \ref{BasicEstimate1},
Corollary \ref{BasicEstimate2} and Corollary \ref{BasicEstimate3} remain valid when the assumption
$w\ge0$ is dropped.
}
\end{remark}
\section{A priori estimates for subdiffusive porous medium type equations} \label{FPME}
We now apply Corollary \ref{BasicEstimate2} to nonnegative weak subsolutions of the equation 
\begin{equation} \label{PMED}
	\partial_t \big(k\ast [u-u_0]\big)-\Delta \big(a u^m\big)= f,\quad (t,x)\in \Omega_T,
\end{equation}
where $m\ge 1$, $\Omega$ is again a (sufficiently smooth) bounded domain in $\iR^N$ and the function $a$ is assumed to satisfy
\begin{equation} \label{acond}
	0<a_1\le a(t,x)\le a_2,\quad \mbox{a.a.}\,(t,x)\in \Omega_T,
\end{equation}
with two constants $a_1,a_2>0$. An important special case is given by $a(t,x)=1$ in $\Omega_T$. We consider \eqref{PMED} together with homogeneous
Dirichlet or Neumann conditions and the initial condition $u|_{t=0}=u_0$ in $\Omega$.

Inserting $w=au^m$ in \eqref{weakform2} we obtain for a.a.\ $t\in (0,T)$ and $n\in \iN$
\begin{align} \label{weakformPME}
	\int_\Omega \Big(\psi \partial_t\big(k_n\ast [u-u_0]\big)+[h_n\ast\nabla (au^m)]\cdot \nabla \psi\Big)\,dx\le \int_\Omega (h_n\ast f)\psi\,dx.
\end{align}
The boundary condition (Dirichlet or Neumann) is encoded by further conditions on $u$ and the space of test functions $\psi$, respectively.
\begin{theorem} \label{PMEthm}
	Let $m\ge 1$, $u_0\in L_{m+1}(\Omega)$, $f\in L_{m+1}(\Omega_T)$ and 
	$a\in L_\infty(\Omega_T)$ such that \eqref{acond} is satisfied.
	Assume that
	one of the following conditions is satisfied.
	\begin{itemize}
		\item[(D)] $u\ge 0$ in $\Omega_T$, $au^m\in L_2((0,T);\mathring{H}^1_2(\Omega))$ and 
		\eqref{weakformPME} holds for all nonnegative $\psi\in \mathring{H}^1_2(\Omega))$. 
		\item[(N)] $u\ge 0$ in $\Omega_T$, $au^m\in L_2((0,T);{H}^1_2(\Omega))$ and 
		\eqref{weakformPME} holds for all nonnegative $\psi\in {H}^1_2(\Omega))$. 
	\end{itemize}
	Assume further that
	$k\ast u\in C([0,T];L_2(\Omega))$ such that $(k\ast u)|_{t=0}=0$.
	
	Then
	\begin{align}
		\int_0^T \int_\Omega  &u^{m+1}\,dx\,dt+
		\frac{1}{a_1}\,\int_0^T \int_\Omega \big[k(T-t)+ k(t)\big]\big| l \astT \nabla (au^m)\big|^2\,dx\,dt\nonumber\\
		& \le C(m) \Big(\frac{a_2}{a_1}\Big)^{m+1}\Big(T\norm{u_0}_{L_{m+1}(\Omega)}^{m+1}
		+\norm{l}_{L_1((0,T))}^{m+1}
		\norm{f}_{L_{m+1}(\Omega_T)}^{m+1}\Big),
        \label{PME_est}
	\end{align}
	where $C(m)=\frac{2(4m)^m}{(m+1)^{m+1}}$.
\end{theorem}
\begin{proof}
	Note first that our assumptions imply that $u\in L_{2m}(\Omega_T)\subset L_{m+1}(\Omega_T)$. Indeed, this follows from $au^m\in L_2((0,T);{H}^1_2(\Omega))\subset L_2(\Omega_T)$ and the assumption that $a$ is bounded away from zero.
	
	We apply Corollary \ref{BasicEstimate2} with $w=au^m$ and obtain the estimate
	\begin{align}
		\int_0^T \int_\Omega a u^{m+1}\,dx\,dt+&
		\frac{1}{2}\,\int_0^T \int_\Omega \big[k(T-t)+ k(t)\big]\big| l \astT \nabla (au^m)\big|^2\,dx\,dt\nonumber\\
		& \le \int_0^T \int_\Omega \big(u_0+l\ast f\big) au^m \,dx\,dt.\label{PME1a}
	\end{align}
	We estimate the integral on the right-hand side using H\"older's and Young's inequality.
	\begin{align}
		\int_0^T &\int_\Omega \big(u_0+l\ast f\big) au^m \,dx\,dt\le a_2
		\int_0^T \int_\Omega \big(|u_0|+l\ast |f|\big) u^m \,dx\,dt\nonumber\\
		& \le a_2 \Big(T^{\frac{1}{m+1}}\norm{u_0}_{L_{m+1}(\Omega)}+\norm{l}_{L_1((0,T))}
		\norm{f}_{L_{m+1}(\Omega_T)}\Big)\norm{u}_{L_{m+1}(\Omega_T)}^{m}\nonumber\\
		& \le \frac{\delta^{-m}}{m+1}a_2^{m+1}\Big(T\norm{u_0}_{L_{m+1}(\Omega)}^{m+1}
		+\norm{l}_{L_1((0,T))}^{m+1}
		\norm{f}_{L_{m+1}(\Omega_T)}^{m+1}\Big)+\frac{2m\delta }{m+1}\norm{u}_{L_{m+1}(\Omega_T)}^{m+1}\label{PME1b},
	\end{align}
	where $\delta>0$ can be chosen arbitrarily. Estimating the first integral on the left of \eqref{PME1a}
	from below by $a_1 \norm{u}_{L_{m+1}(\Omega_T)}^{m+1}$ and combining the resulting inequality with \eqref{PME1b}, where we choose $\delta=\frac{a_1(m+1)}{4m}$, we obtain
	\begin{align*}
		\int_0^T \int_\Omega  &u^{m+1}\,dx\,dt+
		\frac{1}{a_1}\,\int_0^T \int_\Omega \big[k(T-t)+ k(t)\big]\big| l \astT \nabla (au^m)\big|^2\,dx\,dt\nonumber\\
		& \le \frac{2\delta^{-m}}{a_1(m+1)}a_2^{m+1}\Big(T\norm{u_0}_{L_{m+1}(\Omega)}^{m+1}
		+\norm{l}_{L_1((0,T))}^{m+1}
		\norm{f}_{L_{m+1}(\Omega_T)}^{m+1}\Big),
	\end{align*}
	which implies the assertion.
\end{proof}
\begin{remark} {\em 
    In the case of the standard kernel, that is, $(k,l)=(g_{1-\alpha}, g_\alpha)$ with $\alpha\in (0,1)$, Theorem \ref{PMEthm} remains true if the second term on the left-hand side of \eqref{PME_est} is replaced with
    \[\frac{2\gamma_\alpha}{a_1} \int_0^T \int_\Omega |g_{\alpha/2}\ast \nabla (au^m)|^2 \,dx \,dt .\]
    See Example \ref{PME_frac}.
    }
\end{remark}
\begin{remark}
    {\em Theorem \ref{PMEthm} provides an estimate for nonnegative weak subsolutions of \eqref{PMED}. By an analogous argument one can obtain a corresponding result for (possibly sign-changing) weak solutions of 
\begin{equation} \label{PMED2}
	\partial_t \big(k\ast [u-u_0]\big)-\Delta \big(a |u|^{m-1}u\big)= f,\quad (t,x)\in \Omega_T.
\end{equation}
    }
\end{remark}

\section{Uniqueness of weak solutions}
The duality estimates from Section \ref{dualityestimates} are also useful to establish uniqueness results for sub\-diffusive porous medium type problems. 

Let $\Omega$ be a bounded domain in $\iR^N$ and $\Phi:\,\iR\to \iR$ be continuous and strictly increasing with $\Phi(0)=0$. We consider the problem
\begin{equation} \label{subdiffFilt}
	\left\{ \begin{array}{r@{\,=\,}l}
		\partial_t \big(k\ast [u-u_0]\big)-\Delta\big( a \Phi(u)\big)& f,\quad (t,x)\in \Omega_T,\\
		u(t,x) & 0,\quad t\in (0,T),\,x\in \partial\Omega,\\
		u(0,x) & u_0(x),\quad x\in \Omega.
	\end{array} \right.
\end{equation}
\begin{theorem} \label{uniquetheorem}
Let $u_0\in L_2(\Omega)$, and $f\in L_{2}(\Omega_T)$. Then there is at most 
one weak solution $u$ of \eqref{subdiffFilt} in the following sense:
$u\in L_2(\Omega_T)$, $k\ast u\in C([0,T];L_2(\Omega))$ satisfying $(k\ast u)|_{t=0}=0$,
$a\Phi(u)\in L_2((0,T);\mathring{H}^1_2(\Omega))$ and
\begin{align} \label{weakformsubdiffFilt}
	\int_\Omega \Big(\psi \partial_t\big(k_n\ast [u-u_0]\big)+[h_n\ast\nabla (a\Phi(u))]\cdot \nabla \psi\Big)\,dx= \int_\Omega (h_n\ast f)\psi\,dx,
\end{align}
for all $\psi\in \mathring{H}^1_2(\Omega)$ and $n\in \iN$.
\end{theorem}
\begin{proof}
	Suppose that $u_1$ and $u_2$ are weak solutions of \eqref{subdiffFilt}. We set $u=u_1-u_2$,
	$w_i=a\Phi(u_i)$, $i=1,2$, and $w=w_1-w_2$. Then we have
	\begin{equation} \label{subdiffFilt2}
		\left\{ \begin{array}{r@{\,=\,}l}
			\partial_t \big(k\ast u\big)-\Delta w& 0,\quad (t,x)\in \Omega_T,\\
			u(t,x) & 0,\quad t\in (0,T),\,x\in \partial\Omega,\\
			u(0,x) & 0,\quad x\in \Omega,
		\end{array} \right.
	\end{equation}
	in the weak sense. Applying Corollary \ref{BasicEstimate2} we deduce that
	\[
	\int_0^T \int_\Omega wu\,dx\,dt\le 0.
	\]
	On the other hand, since $\Phi$ is increasing and $a\ge a_1>0$, we also have
	\[
	\int_0^T \int_\Omega wu\,dx\,dt=\int_0^T \int_\Omega a\big(\Phi(u_1)-\Phi(u_2)\big)\big(u_1-u_2\big)\,dx\,dt\ge 0.
	\]
	By the strict monotonicity of $\Phi$ it follows that $u_1=u_2$. 
\end{proof}
\begin{remark}
	{\em
		In the case $a=1$, the proof of Theorem \ref{uniquetheorem} can be viewed as the precise subdiffusive analogue of the proof of \cite[Theorem 5.3]{Vaz} on uniqueness of weak solutions for the problem
\begin{equation*}
		\left\{ \begin{array}{r@{\,=\,}l}
			\partial_t u-\Delta \Phi(u)& f,\quad (t,x)\in \Omega_T,\\
			u(t,x) & 0,\quad t\in (0,T),\,x\in \partial\Omega,\\
			u(0,x) & u_0(x),\quad x\in \Omega.
		\end{array} \right.
	\end{equation*}
	}
\end{remark}
\begin{remark} \label{PloRem}
{\em
A uniqueness result for weak solutions to one-dimensional time-fractional porous medium type equations 
\begin{equation} \label{PloProb}
\partial_t^\alpha (u-u_0)-\partial_x \big(D(u)u_x)=0
\end{equation}
with $D\in C^1(\iR)$, $D(0)=0$, and $D(r)>0$ as well as $D'(r)>0$ for $r>0$ was proven
in \cite[Theorem 1]{LoOPR}. Note that by defining $\Phi(r)=\int_0^r D(s)\,ds$, \eqref{PloProb} can be rewritten in our form with $a=1$.  
The basic idea of the argument given in \cite[Theorem 1]{LoOPR} is essentially the same as ours. However, the reasoning for the nonnegativity of the elliptic term is incorrect, the
last equality in \cite[formula (14)]{LoOPR} is not true. The estimate becomes correct, if one uses instead the inequality \eqref {galphaInequality}. Furthermore, the proof in \cite{LoOPR} is not rigorous  
with respect to the required time regularity of the test function. The authors' reference to standard mollification techniques in the book \cite{Vaz} is questionable at this point, since \cite{Vaz} does not consider operators such as the fractional derivative. We would like to point out that standard smoothing methods from classical parabolic theory, such as
Steklov averages, no longer all work in the subdiffusive case,
see \cite{Za-08}.
}
\end{remark}

\section{Subdiffusion-reaction systems} \label{chem}
In this section, we discuss an important application of the duality estimates in the context of subdiffusion-reaction systems. For the sake of simplicity we will restrict ourselves to a single reversible
reaction with two species on both sides. We point out that our results easily extend to general reversible reactions and also to systems of chemical reactions under mass action kinetics with control of mass. We refer to the monograph \cite{Fe19} for a good account of chemical reaction network theory
and to \cite{Fi15,LiMi,MiHaMa,Pie11} for important results on reaction-diffusion systems assuming mass action kinetics.  

Let us consider a reversible chemical reaction
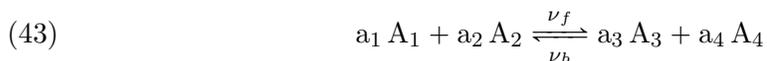
\begin{equation} \label{chemreac}
	\ce{a_1 A_1}+\ce{a_2 A_2 <=>[\nu_f][\nu_b]}\ce{a_3 A_3}+\ce{a_4 A_4}
\end{equation}
involving the chemical species $A_1,\ldots,A_4$ with stoichiometric coefficients $a_1,\ldots,a_4\in \iN$.
$\nu_f$ and $\nu_b$ stand for the reaction constants. Denoting the concentration of $A_i$ by $c_i$ and assuming mass action kinetics the reaction rate of
\eqref{chemreac} is given by
\begin{equation*}
	r(c_1,c_2,c_3,c_4)=-\nu_f \,c_1^{a_1} c_2^{a_2}+\nu_b\, c_3^{a_3} c_4^{a_4}.
\end{equation*}
If the reaction takes place in a bounded domain $\Omega$ and the molecules can diffuse, then
the corresponding reaction-diffusion system (in case of normal diffusion) reads  
\begin{equation} \label{RDclassic}
	\left\{ \begin{array}{r@{\,=\,}l}
		\partial_t c_1-d_1 \Delta c_1 & \;a_1 r(c_1,c_2,c_3,c_4),\\
		\partial_t c_2-d_2 \Delta c_2 & \;a_2 r(c_1,c_2,c_3,c_4),\\
		\partial_t c_3-d_3 \Delta c_3 & -a_3 r(c_1,c_2,c_3,c_4),\\
		\partial_t c_4-d_4 \Delta c_4 & -a_4 r(c_1,c_2,c_3,c_4),
	\end{array} \right.
\end{equation}
where $d_1,\ldots,d_4$ are the (in general different) diffusion coefficients. It is natural to consider \eqref{RDclassic} together with homogeneous Neumann boundary conditions
\begin{equation} \label{NBC}
	\partial_\nu c_i=0,\quad t>0,\,x\in \partial\Omega,\,i=1,\ldots,4.
\end{equation}
Assuming that the initial data $c_i^0$ (taken e.g.\ from $L_\infty(\Omega)$) are nonnegative, the solution of the reaction-diffusion system takes values in $[0,\infty)^4$ as long as it exists.
Short-time existence and uniqueness of smooth solutions is well-known. The question of global smooth
solutions in the general case is still an open problem. However, for quadratic reaction terms (that is, $a_i=1$ for $i=1,\ldots,4$), one has global existence of smooth solutions for arbitrary data 
and in any space dimension (\cite{CaGoVa,FeMoTa}). Global existence of renormalized solutions was established in \cite{Fi15}. If the reaction terms are integrable, then the renormalized solutions become weak solutions defined in the usual way. Concerning global weak
resp.\ generalized solutions we also refer to \cite{CaDeFe,DeFePiVo,LaPe,LaPi,LanWi}.

Duality estimates play an important role in deriving {\em a priori} estimates for the concentrations $c_i$ in
$L_p$-spaces. For example, in the case of \eqref{RDclassic} one can obtain a bound in $L_2(\Omega_T)$ for solutions
existing on the time-interval $[0,T]$, see e.g.\ \cite[Section 6]{Pie11} and \cite{DeFePiVo, LaPi}. The idea is to consider a suitable linear combination of the concentrations. Setting
\begin{align}
	u&:=a_3 c_1+a_4 c_2+a_1 c_3+a_2 c_4, \label{udef}\\
	u_0&:=a_3 c_1^0+a_4 c_2^0+a_1 c_3^0+a_2 c_4^0,\label{u0def}
\end{align}
and
\[
a:=\frac{d_1 a_3 c_1+d_2 a_4 c_2+d_3 a_1 c_3+d_4 a_2 c_4}{a_3 c_1+a_4 c_2+a_1 c_3+a_2 c_4}
\]
we find that
\begin{equation} \label{sumeq}
	\partial_t u-\Delta\big(a u)=0. 
\end{equation}
Observe that
\[
0<\min\limits_{i=1,\ldots,4} d_i\le a(t,x)\le \max\limits_{i=1,\ldots,4} d_i.
\]
Hence by the $L_2$-estimate for \eqref{localPDE} in the linear case $m=1$ one gets a bound
\[
\norm{u}_{L_2(\Omega_T)}\le C\norm{u_0}_{L_2(\Omega)},
\] 
which, thanks to the positivity of each $c_i$, yields an $L_2$-bound for each species.

This line of arguments extends to more general diffusion operators, including degenerate diffusion, see e.g.\ \cite{LaPi}. For example, if we replace in \eqref{RDclassic} the diffusion terms $d_i \Delta c_i$ by
$d_i \Delta \big(c_i^m\big)$ with some $m>1$, $i=1,\ldots,4$, then we obtain an estimate
\[
\norm{u}_{L_{m+1}(\Omega_T)}\le C\norm{u_0}_{L_{m+1}(\Omega)}.
\] 
Note that here
\begin{equation} \label{anew}
	a=\frac{d_1 a_3 c_1^m+d_2 a_4 c_2^m+d_3 a_1 c_3^m+d_4 a_2 c_4^m}{\big(a_3 c_1+a_4 c_2+a_1 c_3+a_2 c_4\big)^m}
\end{equation}
is again bounded from above and bounded away from zero. 

The mathematical modelling of chemical reactions in systems with anomalous diffusion is a challenging
task (\cite{Nepo}). According to \cite[page 2]{Nepo}, ``there is no universal model for reactions in a subdiffusive system; the description depends on the details of the underlying physics.''
If we start with the pure subdiffusion equation for a species with concentration $c$ in the
form with $(k,l)=(g_{1-\alpha},g_{\alpha})$, $\alpha\in (0,1)$,
\[
\partial_t c-\partial_t^{1-\alpha}d\Delta c=0,
\]
see \cite{Metz}, then simply adding a reaction term $R(c)$ on the right-hand side (that is the reaction rate is independent of the diffusion) may lead to a physically inconsistent equation, one problem being that preservation of positivity may fail. We refer to \cite{Frö,HeLaWe,Law} for more details.  
Physically meaningful models include equations of the form
\[
\partial_t c-\partial_t^{1-\alpha}d\Delta c=\partial_t^{1-\alpha} R(c),
\]
which is equivalent to
\begin{equation} \label{ReactionSubdiff}
	\partial_t^\alpha(c-c_0)-d\Delta c=R(c).
\end{equation}
See the discussion in \cite{Nepo} and the derivation of such an equation for a specific reaction in 
\cite{SeWoTa}.

Let us consider the reaction \eqref{chemreac} in a subdiffusive regime. Using the form \eqref{ReactionSubdiff} with a more general $\mathcal{PC}$-kernel $k$ and
allowing also for a porous medium type diffusion operator ($m\ge 1$), the RD-system \eqref{RDclassic} turns into the system
\begin{equation} \label{RDnew}
	\left\{ \begin{array}{r@{\,=\,}l}
		\partial_t \big(k\ast [c_1-c_1^0]\big) -d_1 \Delta c_1^m & \;a_1 r(c_1,c_2,c_3,c_4),\\
		\partial_t \big(k\ast [c_2-c_2^0]\big)-d_2 \Delta c_2^m & \;a_2 r(c_1,c_2,c_3,c_4),\\
		\partial_t \big(k\ast [c_3-c_3^0]\big)-d_3 \Delta c_3^m & -a_3 r(c_1,c_2,c_3,c_4),\\
		\partial_t \big(k\ast [c_4-c_4^0]\big)-d_4 \Delta c_4^m & -a_4 r(c_1,c_2,c_3,c_4).
	\end{array} \right.
\end{equation}
We have the following result.
\begin{theorem} \label{subdiffsystemthm}
	Let $m\ge 1$ and $c_i^0\in L_{\infty}(\Omega)$ with $c_i^0\ge 0$, $i=1,\ldots,4$. Suppose that $(c_1,c_2,c_3,c_4)$
	is a sufficiently smooth solution of \eqref{RDnew} (together with \eqref{NBC}) on $(0,T]\times \Omega$ satisfying the initial conditions $c_i|_{t=0}=c_i^0$ in $\Omega$ for $i=1,\ldots,4$.
	Assume further that $c_i\ge 0$ in $\Omega_T$ for all $i=1,\ldots,4$.
	Then
	\begin{equation}
		\sum_{i=1}^4 \norm{c_i}_{L_{m+1}(\Omega_T)}\le C \sum_{i=1}^4 \norm{c_i^0}_{L_{m+1}(\Omega)},
	\end{equation}
	where the constant $C$ depends only on $m$, $T$, $\norm{l}_{L_1((0,T))}$, $a_i$, and the diffusion coefficients
	$d_i$, $i=1,\ldots,4$. 
\end{theorem}
\begin{proof}
	Consider $u$ and $u_0$ defined in \eqref{udef} and \eqref{u0def}, respectively. Then we have
	\[
	\partial_t \big(k\ast [u-u_0]\big)- \Delta \big(a u^m\big)=0\quad \mbox{in}\;\Omega_T,
	\]
	where $a$ is given by \eqref{anew}. As already noted before, $a$ is bounded from above and bounded away from zero, say $0<a_1\le a(t,x)\le a_2$. Thus we may apply Theorem \ref{PMEthm}, which yields an estimate 
	\[
	\norm{u}_{L_{m+1}(\Omega_T)}\le C \norm{u_0}_{L_{m+1}(\Omega)},
	\]
	where $C=C(m, T, |l|_{L_1((0,T))},a_1,a_2)$. The assertion now follows immediately as $c_i\ge 0$ in $\Omega_T$, $i=1,\ldots,4$.
\end{proof}
We remark that there is a vast literature on semilinear time-fractional (and more general subdiffusive) parabolic problems. We refer to \cite{GaWa} and the references therein. 

{\em Preservation of positivity.} In Theorem \ref{subdiffsystemthm}, we assumed that $c_i\ge 0$ in $\Omega_T$ for all $i=1,\ldots,4$. We strongly believe that this assumption can be dropped as one can expect system \eqref{RDnew}
to be positivity preserving. Concerning systems of subdiffusion-reaction equations such as \eqref{RDnew}, there appear to be no results on this property in the literature, even in the linear time-fractional case.

We give some indications as to why preservation of positivity is to be expected for \eqref{RDnew}, at least under some additional regularity assumptions. Assume that $k$ is completely monotone. 
Then, from Proposition \ref{comoprop}, we know $\lim_{t\to 0+}tk(t)=0$ and that $t \dot{k}(t)$ is locally integrable
on $[0,\infty)$.

Let $v:\,[0,T]\to \iR$ be sufficiently smooth, e.g.\ $v\in C^1([0,T])$. Using integration by parts, we find for $t\in (0,T)$ that
\begin{align}
	\frac{d}{dt}\big(k\ast [v-v(0)]\big)(t)& = \big(k\ast \dot{v}\big)(t)
	= \int_0^t k(t-\tau)\frac{d}{d\tau}\big(v(\tau)-v(t)\big)\,d\tau\nonumber\\
	&= k(t)\big(v(t)-v(0)\big)+\int_0^t [-\dot{k}(t-\tau)]\big(v(t)-v(\tau)\big)\,d\tau,\label{reform}
\end{align}
where we use that $k(t-\tau)(v(t)-v(\tau))\to 0$ as $\tau\to t-$. The identity \eqref{reform}
already appeared in \cite{AlCaVa,Lu} for the standard kernel $k=g_{1-\alpha}$, $\alpha\in (0,1)$.

Let now $c=(c_1,c_2,c_3,c_4)$ be a sufficiently smooth solution of \eqref{RDnew} on $(0,T]\times \Omega$, where we want to assume that the initial data $c_i^0$ are smooth and satisfy $c_i^0>0$ in all of $\bar{\Omega}$. Suppose
there exists a first time $t_*\in (0,T]$ where one of the components of  $c$, say the $i$-th one, becomes zero. Suppose further that this takes place at a point $(t_*,x_*)$ with $x_*\in \Omega$.
Then $c_i(t_*,x_*)=0$
and $c_j(t,x)\ge 0$ for all $j\in \{1,\ldots,4\}$, $t\in [0,t_*]$ and $x\in \bar{\Omega}$. This implies
$-\Delta(c_i^m)(t_*,x_*)\le 0$ and (using \eqref{reform})
\begin{align*}
	\partial_t \big(k\ast [c_i-c_i^0]\big)(t_*,x_*)& = k(t_*)\big(c_i(t_*,x_*)-c_i^0(x_*)\big)\\
	&\quad
	+\int_0^{t_*} [-\dot{k}(t_*-\tau)]\big(c_i(t_*,x_*)-c_i(\tau,x_*)\big)\,d\tau\\
	&\le -k(t_*)c_i^0(x_*)<0.
\end{align*}
Hence
\[
\partial_t \big(k\ast [c_i-c_i^0]\big)(t_*,x_*)-\Delta(c_i^m)(t_*,x_*)<0.
\]
On the other hand, the right-hand side of \eqref{RDnew} is quasipositive, that is, its $i$-th component
is nonnegative at $(t_*,x_*)$ because $c_i(t_*,x_*)=0$ and $c_j(t_*,x_*)\ge 0$ for all $j\neq i$.
This contradiction shows that the first vanishing of one of the components of $c$ can only happen at
$\partial\Omega$. If a suitable Hopf type lemma were available, one could use the homogeneous Neumann boundary condition to arrive again at a contradiction and thus show that all components of $c$ remain strictly positive. However, appropriate Hopf type lemmas seem to be missing in the literature. 

We conclude this section with an additional estimate in terms of entropy variables in the linear case. So let us assume that $m=1$ in \eqref{RDnew}. For the sake of simplicity we also assume that $\nu_f=\nu_b$, which ensures that $c_{*}=(1,1,1,1)$ is a homogeneous equilibrium of \eqref{RDnew}. Let $c_i(t,x)> 0$ and sufficiently smooth, $i=1,\ldots,4$. We introduce the entropy variables $z_i=\phi(c_i)$, $i=1,\ldots,4$, where
\[
\phi(s)=s\log s-s+1, \quad s> 0.
\]
Note that $\phi'(s)=\log s$ and $\phi(s)\ge 0$.
Multiplying the $i$-th equation in \eqref{RDnew} by $\log c_i$ we obtain
that  
\[
\log c_i\, \partial_t \big(k\ast [c_i-c_i^0]\big) -d_i \Delta z_i \le \mbox{sig}(i) a_i r(c_1,c_2,c_3,c_4)\log c_i,
\]
where sig$(i)=1$ for $i=1,2$ and sig$(i)=-1$ for $i=3,4$. In fact,
\[
\Delta z_i=\Delta\, \phi(c_i)=\log(c_i)\Delta c_i+\frac{1}{c_i}|\nabla c_i|^2.
\]
Since $\phi$ is convex, we can use \cite[Corollary 6.1]{KSVZ} (which also holds for more singular kernels provided $c_i$ is sufficiently smooth as we assume; alternatively one can regularize the kernel as above) to see that
\[
\log c_i\, \partial_t \big(k\ast [c_i-c_i^0]\big)\ge  \partial_t \big(k\ast [\phi(c_i)-\phi(c_i^0)]\big).
\]
Therefore
\[
\partial_t \big(k\ast [z_i-z_i^0]\big) -d_i \Delta z_i \le \mbox{sig}(i) a_i r(c_1,c_2,c_3,c_4)\log c_i,
\]
with $z_i^0=\phi(c_i^0)$. Summing over all $i$ and using that
\[
\sum_{i=1}^4 \mbox{sig}(i) a_i r(c_1,c_2,c_3,c_4)\log c_i\le 0,
\]
see e.g.\ \cite{Fe19, MiHaMa}, we finally obtain
\[
\partial_t \big(k\ast [z-z^0]\big) -\Delta( d z) \le 0,
\]
where
$z=\sum_{i=1}^4 z_i$, $z_0=\sum_{i=1}^4 z_i^0$, and $d=z^{-1} \sum_{i=1}^4 d_i z_i$.
Applying Theorem \ref{PMEthm}, we find the following result.
\begin{theorem} \label{entropyvar}
    Let $c_i^0\in L_{\infty}(\Omega)$ with $c_i^0> 0$, $i=1,\ldots,4$. Suppose that $(c_1,c_2,c_3,c_4)$
	is a sufficiently smooth solution of \eqref{RDnew} with $m=1$ (together with \eqref{NBC}) on $(0,T]\times \Omega$ satisfying the initial conditions $c_i|_{t=0}=c_i^0$ in $\Omega$ for $i=1,\ldots,4$.
	Assume further $\nu_f=\nu_b$ and that $c_i> 0$ in $\Omega_T$ for all $i=1,\ldots,4$.
	Then
	\begin{equation} \label{entest}
		\sum_{i=1}^4 \norm{z_i}_{L_{2}(\Omega_T)}\le C \sum_{i=1}^4 \norm{z_i^0}_{L_{2}(\Omega)},
	\end{equation}
	where the constant $C$ depends only on $m$, $T$, $\norm{l}_{L_1((0,T))}$ and the diffusion coefficients
	$d_i$, $i=1,\ldots,4$. 
\end{theorem}
\begin{remark}
    {\em The estimate \eqref{entest} is a subdiffusive analogue of a corresponding estimate from classical reaction-systems, see e.g.\ \cite[Section 2.3]{DeFe} and the references given therein. Note that \eqref{entest} improves the estimate
    from Theorem \ref{subdiffsystemthm} inasmuch as it provides a bound for the $c_i$ in $L_2(\log L)_2(\Omega_T)$.
    }
\end{remark}
\section{Abstract duality estimate and further applications}\label{absv}

In this section, we will present an abstract version of the duality inequalities (for solutions) from Section \ref{dualityestimates}. Instead of the Laplacian, we consider
an operator which is induced by a symmetric bilinear form on some Hilbert space. To illustrate the main result, we derive \textit{a priori} estimates for a fully nonlocal porous medium type equation. 

We start by recalling the setting used in \cite{ZWH}. Let $\mathcal{V}$ and $\mathcal{H}$ be real separable Hilbert spaces such that $\mathcal{V}$ is densely and continuously embedded into $\mathcal{H}$. Identifying $\mathcal{H}$ with its dual $\mathcal{H}^{\prime}$ we have $\mathcal{V} \hookrightarrow \mathcal{H} \hookrightarrow \mathcal{V}^{\prime}$, and
$$
(h, v)_{\mathcal{H}}=\langle h, v\rangle_{\mathcal{V}^{\prime} \times \mathcal{V}}, \quad h \in \mathcal{H}, v \in \mathcal{V},
$$
where $(\cdot, \cdot)_{\mathcal{H}}$ and $\langle\cdot, \cdot\rangle_{\mathcal{V}^{\prime} \times \mathcal{V}}$ denote the scalar product in $\mathcal{H}$ and the duality pairing between $\mathcal{V}^{\prime}$ and $\mathcal{V}$, respectively.

Notice that the arguments employed in the proof of \eqref{dualprop} can be easily adapted to see that
\begin{equation}\label{dualdual}
	\int_0^T\langle f(t) , (h\astT g)(t)\rangle_{\mathcal{V}^{\prime} \times \mathcal{V}} \,dt= \int_0^T\langle (h\ast f)(t) , g(t)\rangle_{\mathcal{V}^{\prime} \times \mathcal{V}} \,dt , 
\end{equation}
for all $f \in L_p\left([0, T] ; \mathcal{V}^{\prime}\right)$, $g \in L_{p'}\left([0, T] ; \mathcal{V}\right)$, $\frac{1}{p} + \frac{1}{p'} = 1, 1\leq p\leq \infty$, and $h \in L_1\left((0, T)\right)$.

Now, consider the following problem for the pair of functions $(u,w)$
\begin{equation}\label{abs}
	\frac{d}{d t}([k *(u-u_0)](t), v)_{\mathcal{H}} + \mathfrak{a}(w(t), v) = \langle f(t), v\rangle_{\mathcal{V}^{\prime} \times \mathcal{V}}, \quad v \in \mathcal{V}, \text { a.a. } t \in(0, T)	 ,
\end{equation}
where $d/dt$ means the generalized derivative of real functions on $(0,T)$, $\mathfrak{a}:\mathcal{V} \times \mathcal{V} \rightarrow \R$ is a bilinear symmetric form that satisfies
\begin{align}
	|\mathfrak{a}(\tilde{v},v) | \leq M\|\tilde{v}\|_\mathcal{V}  \|v\|_\mathcal{V} ,\nonumber\\
    \mathfrak{a}(v,v) \geq \kappa\|v\|_\mathcal{V}^2 -d\|v\|_\mathcal{H}^2 , \label{Ha}
\end{align}
for all $\tilde{v},v\in \mathcal{V}$, with constants $M,\kappa>0, d\geq0$. In addition, $u_0 \in \mathcal{H}$ and $f \in L_2\left((0, T) ; \mathcal{V}^{\prime}\right)$ are given data. Here, $ u\in L_2((0,T);\mathcal{H})$, $k\ast u\in C([0,T];\mathcal{H}) \cap H^{1}_2((0,T);\mathcal{V}')$, and $(k\ast u)|_{t=0}=0$, and $w\in L_2((0,T);\mathcal{V})$. The kernel $k$ is again of $\mathcal{PC}$-type.

 By following the argument from Section \ref{dualityestimates}, we obtain the subsequent equivalent weak formulation of \eqref{abs} involving the more regular kernel $k_n$. For any $\psi \in \mathcal{V}$ and a.a.\ $t\in(0,T)$, we have
	\begin{equation}\label{weakform3}
		\left(\psi , \frac{d}{d t}\big(k_n *\left[u-u_0\right]\big)\right)_\mathcal{H}(t)+\mathfrak{a}\big(\left(h_n * w\right)(t) , \psi\big) = \big\langle \left(h_n * f\right)(t) , \psi \big\rangle_{\mathcal{V}^{\prime} \times \mathcal{V}},
	\end{equation}
	for all $ n \in \mathbb{N}$.

In view of \eqref{Ha}, the shifted form $\tilde{\mathfrak{a}}(w, v): = \mathfrak{a}(w, v) + d(w, v)_{\cH}$ is nonnegative. Let $A:\,D(A)\to \cH$ be the operator on $\cH$ (with maximal domain) associated with $\tilde{\mathfrak{a}}$, that is, $\tilde{\mathfrak{a}}(w, v) = (Aw,v)_{\cH}$ for all $w\in D(A)$ and $v\in \cV$. It is well known that
$A$ is a nonnegative self-adjoint operator, which admits a (self-adjoint) square root $A^{\frac{1}{2}}$, with domain $D(A^{\frac{1}{2}}) := \mathcal{V}$ and
\begin{equation}\label{square}
	\tilde{\mathfrak{a}}(w,v) =  \left(A^{\frac{1}{2}}w , A^{\frac{1}{2}}v\right)_\mathcal{H},\quad w,v\in \cV,
\end{equation}
see, e.g., \cite{Kato,McIntosh}.
\begin{lemma} \label{BasicEstimateAb}
	Let $u_0\in \mathcal{H}$, $f\in L_2((0,T),\mathcal{V}')$, $w\in L_2((0,T),\mathcal{V})$, $u\in L_2((0,T);\cH)$, and $k\ast u\in C([0,T];\mathcal{H})$ such that $(k\ast u)|_{t=0}=0$. Assume $k$ is completely monotone. If the pair $(u,w)$ satisfies \eqref{abs}, then
	\begin{align}
		& \int_0^T \left( w , u \right)_\mathcal{H} dt + 
		\frac{1}{2}\,\int_0^T \big[k(T-t)+ k(t)\big]\norms{l \astT A^{\frac{1}{2}} w}^2_{\mathcal{H}}(t)\,dt
		\nonumber\\
		+ &\frac{1}{2}\int_0^T \int_0^t \frac{1}{s} [k(s) - k(2s)]
		\norm{(l \astT A^{\frac{1}{2}} w)(t,\cdot)-(l \astT A^{\frac{1}{2}} w)(t-s,\cdot)}^2_{\mathcal{H}}ds\,dt \nonumber\\
		& \leq \int_0^T \left( u_0  ,w\right)_{\mathcal{H}} + \int_0^T \langle l\ast f , w\rangle_{\cV'\times \cV} \, dt + d\int_0^T \left( w  , l\ast w\right)_{\mathcal{H}} \, dt . \label{BasEst3ab}
	\end{align}
\end{lemma}
\begin{proof}
	For $t\in (0,T)$ we choose in \eqref{weakform3} $\psi=(l \astT w)(t,\cdot)$ and integrate in time over $(0,T)$. From the duality property \eqref{dualprop} and the definition of $h_n$ we infer that
	\begin{align*}
		\int_0^T \left((l \astT w) , \frac{d}{dt}\big(k_n\ast [u-u_0]\big) \right)_\mathcal{H}  dt =
		\int_0^T  \left( w , \big(h_n\ast [u-u_0]) \right)_{\mathcal{H}} dt.
	\end{align*}
	We write
	$$\mathfrak{a}(h_n\ast w, l\astT w) = \mathfrak{a}(h_n\astT w, l\astT w) + \mathfrak{a}(h_n\ast w - h_n\astT w, l\astT w) .$$
	As in the scalar case, we can use the identity $f=-\frac{d}{dt} (1\astT f)$, together with \eqref{succprop}, to see that
	\begin{align*}
h_n\astT w=-\frac{d}{dt} \big[(k\ast h_n)\astT (l \astT\nabla w)\big]\quad\; \mbox{in}\;L_2((0,T);\cV).
\end{align*}
We shift the bilinear form $\mathfrak{a}$ as
\begin{align*}
    \mathfrak{a}\left(h_n\astT w, l\astT w\right) 
    & = \tilde{\mathfrak{a}}\left(h_n\astT w, l\astT w\right)
    -d \left(h_n\astT w, l\astT w\right)_{\cH}
    \end{align*}
Using \eqref{dualpropp} (where $\tilde{\mathfrak{a}}$ is regarded as an inner product for $\cV$), \eqref{square}, and Lemma \ref{fundlemma1} we   have
    \begin{align*}
    \int_0^T \tilde{\mathfrak{a}}& \left(h_n\astT w, l\astT w\right)\,dt 
		= - \int_0^T \tilde{\mathfrak{a}}\left( \frac{d}{dt} \big[k_n\astT  (l \astT  w)\big], (l \astT w)\right)\,dt\\
        &= \int_0^T  \tilde{\mathfrak{a}}\left(  (l \astT w),\frac{d}{dt} \big[k_n\ast  (l \astT  w)\big]\right)\,dt\\
		& =  \int_0^T \left( (l \astT A^{\frac{1}{2}}w),\frac{d}{dt} \big[k_n\ast (l \astT  A^{\frac{1}{2}}w)\big] \right)_\mathcal{H}\,dt\\
		& =  \int_0^T \Big(\frac{1}{2}\,\frac{d}{dt}\Big(k_n\ast \norms{l \astT A^{\frac{1}{2}} w}^2_{\mathcal{H}}\Big)(t)+
		\frac{1}{2}\,k_n(t)\norms{(l \astT A^{\frac{1}{2}} w)}^2_{\mathcal{H}}(t)\Big)\,dt\nonumber\\
		& \quad+\int_0^T \frac{1}{2}\int_0^t [-\dot{k}_n(s)]
		\norm{(l \astT A^{\frac{1}{2}} w)(t,\cdot)-(l \astT A^{\frac{1}{2}} w)(t-s,\cdot)}^2_{\mathcal{H}}\,ds\,dt\\
        & \ge \frac{1}{2}\,\int_0^T \big[k_n(T-t)+ k_n(t)\big]\norms{l \astT A^{\frac{1}{2}} w}^2_{\mathcal{H}}(t)\,dt\\
& \quad +\frac{1}{2}\int_0^T \int_0^t\frac{1}{s} [k_n(s) - k_n(2s)]\norm{(l \astT A^{\frac{1}{2}} w)(t,\cdot)-(l \astT A^{\frac{1}{2}} w)(t-s,\cdot)}^2_{\mathcal{H}}\,ds\,dt.
	\end{align*}
Sending now $n\to \infty$ leads to the desired inequality.
\end{proof}
\begin{remark}
    {\em (i) If we replace, in Lemma \ref{BasicEstimateAb}, the assumption \eqref{Ha} with the nonnegativity condition $\mathfrak{a}(v,v)\ge 0$ for all $v\in \cV$, then arguing as before but without shifting the bilinear form we can show that
   \begin{align}
		\int_0^T \left( w , u \right)_\mathcal{H} dt
		& \leq \int_0^T \left( u_0  ,w\right)_{\mathcal{H}} \, dt + \int_0^T \langle l\ast f , w\rangle_{\cV'\times \cV} \, dt. \label{BasEst3abZero}
	\end{align}  
Note that this estimate holds for every $\mathcal{PC}$-kernel $k$.

    (ii) If $f$ belongs to $L_2((0,T);\cH)$, the duality brackets can be replaced with the inner product in $\cH$.
    }
\end{remark}

\begin{example}
{\em     
We will illustrate through a fully nonlocal porous medium equation that Lemma \ref{BasicEstimateAb} can be applied to more general problems. More precisely, we consider the following extension of problem \eqref{PMED} (now on the full space):
\begin{equation} \label{2NPMED}
	\partial_t \big(k\ast [u-u_0]\big) + (-\Delta)^{\beta/2} \big(a u^m\big)= f,\quad (t,x)\in (0,T)\times\R^N,
\end{equation}
where $m\ge 1$, the function $a$ satisfies \eqref{acond} with $\Omega=\iR^N$, and $\beta\in (0,2]$. The kernel $k$ is again assumed to be a $\mathcal{PC}$-kernel, as for all PDE problems in the present paper. The fractional Laplacian can be defined via Fourier transform:
$$ \widehat{(-\Delta)^{\beta/2}\varphi}(\xi) = |\xi|^\beta\hat{\varphi}(\xi),$$
e.g.\ for $\varphi \in \mathcal{S}$, where $\mathcal{S}$ stands for the Schwartz class. This definition can be extended to suitable larger function spaces such as Bessel potential spaces of order $\beta$. We refer to \cite{DPV,Kwa} for a good account on the fractional Laplacian and its mapping properties.  

Existence, uniqueness, and regularity results for \eqref{2NPMED} were obtained in \cite{deP-Qui-Rod-Vaz-11, deP-Qui-Rod-Vaz-12} for the case with $a=1$,
$f=0$, $m>m_*:=(N-\beta)_+/N$, and the usual time derivative.

Equation \eqref{2NPMED} can be rewritten in the form of \eqref{abs}
with $w=au^m$,  $\cH = L_2(\R^N)$, $\cV = H^{\beta / 2} _2\left(\mathbb{R}^N\right)$ and the bilinear form
\begin{equation*}
	\mathfrak{a}(w,v) : = \int_{\mathbb{R}^N} \big((-\Delta)^{\beta / 4} w\big) \big((-\Delta)^{\beta / 4} v\big) \, dx.
\end{equation*}

Let $u_0\in L_2(\R^N)$, $f \in L_2((0,T); L_2(\R^N))$ and $w\in L_2((0,T); H_2^{\beta/2}(\R^N))$, and $k\ast u\in C([0,T];L_2(\R^N)$ such that $(k\ast u)|_{t=0}=0$. For a solution $u$ of \eqref{2NPMED} in the form \eqref{weakform3}, we can apply \eqref{BasEst3abZero} to yield
\begin{equation}\label{estim_weak}
	\int_0^T \left( w , u \right)_{L_2} dt   \leq \int_0^T \left( w , u_0 + l\ast f \right)_{L_2} dt .
\end{equation}
Taking $w=au^m$ and applying H\"older's inequality then yields
\begin{equation*}
		a_1\int_0^T \|u(t)\|_{L_{m+1}(\iR^N)}^{m+1} dt   \leq 
        \norm{w}_{L_2((0,T)\times \iR^N)}\norm{u_0}_{L_2(\iR^N)}+
        \norm{l}_{L_1((0,T))}\norm{f}_{L_2((0,T)\times \iR^N)},
 \end{equation*}       
which shows $u\in L_{m+1}((0,T)\times \R^N)$. Suitable additional assumptions on the data $u_0$ and $f$ allow us to show analogous estimates as in \eqref{PME1b}. In fact, $\norm{u}_{ L_{m+1}((0,T)\times\R^N)}$ can be bounded from above purely by the data if 
$u_0\in L_{m+1}(\R^N)$ and  $f\in L_{m+1}((0,T)\times\R^N)$.
}
\end{example}

\end{document}